\documentclass{amsart}
\usepackage[initials]{amsrefs}
\usepackage{amssymb}
\newtheorem{theorem}{Theorem}[section]
\newtheorem{remark}[theorem]{Remark}

\newtheorem{definition}[theorem]{Definition}

\newtheorem{corollary}[theorem]{Corollary}

\DeclareMathOperator*{\SZ}{\mathcal{SZ}}
\DeclareMathOperator{\card}{card}

\begin{document}

\title{An undecidable case of lineability in \({\mathbb R}^{\mathbb R}\)}

\author{Jos\'{e} L.\ G\'{a}mez-Merino \and Juan B. Seoane-Sep\'{u}lveda}

\address{Plaza de Ciencias 3,\newline\indent  Facultad de Ciencias Matem\'{a}ticas,\newline\indent  Departamento de An\'{a}lisis Matem\'{a}tico,\newline \indent Universidad Complutense de Madrid,\newline\indent  Madrid 28040, Spain.}
\email{jlgamez@mat.ucm.es and jseoane@mat.ucm.es}

\thanks{The authors were supported by the Spanish Ministry of Science and Innovation, grant MTM2009-07848.}


\begin{abstract}
Recently it has been proved that, assuming that there is an almost disjoint family of cardinality \(2^{\mathfrak c}\) in \(\mathfrak c\) (which is assured, for instance, by either Martin's Axiom, or CH, or even \mbox{$2^{<\mathfrak c}=\mathfrak c$}) one has that the set of Sierpi\'nski-Zygmund functions is \(2^{\mathfrak{c}}\)-strongly algebrable (and, thus, \(2^{\mathfrak{c}}\)-lineable). Here we prove that these two statements are actually equivalent and, moreover, they both are undecidable. This would be the first time in which one encounters an undecidable proposition in the recently coined theory of lineability and spaceability.
\end{abstract}

\maketitle

\section{Preliminaries and background}

As a consequence of the classic Luzin's Theorem we have that for every measurable function \(f\colon\mathbb R\to\mathbb R\), there is a measurable set \(S\subset\mathbb R\), of infinite measure, such that \(f\vert_S\) is continuous. A natural question would be whether similar results could be obtained for arbitrary functions (not necessarily measurable). In other words, given any arbitrary function \(f\colon\mathbb R\to\mathbb R\), can we find a ``large'' subset \(S\subset\mathbb R\) for which \(f\vert_S\) is continuous? In 1922, Blumberg \cite{B_1922} provided an affirmative answer to this question.

\begin{theorem}[Blumberg, 1922]
Let \(f\colon\mathbb R\to\mathbb R\) be an arbitrary function. There exists a dense subset \(S\subset\mathbb R\) such that the function \(f\vert_S\) is continuous.
\end{theorem}

Blumberg's proof of his theorem (see, e.g., \cite{K_2006}*{p.\ 154}) shows that the set \(S\) above is countable. Of course, we could wonder whether we can choose the subset \(S\) in Blumberg's theorem to be uncountable. A (partial) negative answer was given in \cite{SZ_1923} by Sierpi\'nski and Zygmund.

\begin{theorem}[Sierpi\'{n}ski-Zygmund, 1923]\label{exa:SZ}
There exists a function \(f\colon\mathbb R\to\mathbb R\) such that, for any set \(Z\subset\mathbb R\) of cardinality the continuum, the restriction \(f\vert_Z\) is not a Borel map (and, in particular, not continuous.)
\end{theorem}

From now on, we shall say that a function \(f\colon\mathbb R\to\mathbb R\) is a \emph{Sierpi\'nski-Zygmund function} if~it~satisfies the condition in Sierpi\'nski-Zygmund's Theorem, and we denote
    \[
    \SZ=\{\,f\colon\mathbb R\to\mathbb R \,:\, f\text{ is a Sierpi\'nski-Zygmund function}\,\}.
    \]

Before continuing, let us recall some recent well known concepts that shall be useful in what follows:

\begin{definition}
Given a certain property we say that the subset $M$ of a topological vector space $X$ which satisfies it is \emph{$\mu$-lineable} if $M \cup \{0\}$ contains a vector space of dimension $\mu$ (finite or infinite cardinality). Also, if we let $\mathcal{L}$ be an algebra, we say that a set $A \subset \mathcal{L}$ is $\beta$-algebrable if there exists an algebra $\mathcal{B}$ so that $\mathcal{B} \subset A \cup \{ 0\}$ and card$(Z) = \beta$, where  $\beta$ is a cardinal number and $Z$ is a minimal system of generators of $\mathcal{B}$. We say that a subset $E$ of a commutative linear algebra $B$ is {\em strongly $\kappa$-algebrable} if there exists a $\kappa$-generated free algebra $A$ contained in $E \cup \{0\}$.
\end{definition}

We refer the interested reader to \cite{AGS_2005,APS_2006,AS_2007,BGPS_PAMS,BFPS_2012,EGS_TAMS,GMSS_2010,GMS_2010,GGMS_2010,S_PhD_thesis} for recent developments on the previous concepts, where many examples are given and techniques are developed. Next, and coming back to the class $\SZ$, let us recall some known results about this class.

\begin{enumerate}
\item It is known that if the Continuum Hypothesis (CH) holds then the restriction of a function in \(\SZ\) to any uncountable set can not be continuous (see, e.g., \cite{K_2006}*{pp.\ 165,\,\nolinebreak 166}).
\item Also, CH is necessary in this frame. Shinoda proved in 1973 \cite{S_1973} that if Martin's Axiom and the negation of CH hold then, for every \(f\colon\mathbb R\to\mathbb R\), there exists an uncountable set \(Z\subset\mathbb R\) such that \(f\vert_Z\) is continuous.
\item The functions in $\SZ$ are never measurable and, although it is possible to construct them being injective, they are nowhere monotone in a very strong way. (Their restriction to any set of cardinality $\mathfrak{c}$ is not monotone.)
\item In 1997, Balcerzak, Ciesielski, and Natkaniec showed in \cite{BCN_1997} that, assuming the set-theoretical condition \(\operatorname{cov}(\mathcal M)=\mathfrak c\) (which is true under Martin's Axiom or CH), there exists a Darboux function that is in $\SZ$ as well. They prove also that there exists a model of ZFC in which there are no such functions (see, also, \cite{P_2002,CN_1997,CN_1997_2}).
\item Later, G\'{a}mez-Merino, Mu\~{n}oz-Fern\'{a}ndez, S\'{a}nchez, and Seoane-Sep\'{u}lveda (2010) proved in \cite{GMSS_2010}*{Theorems 5.6 and 5.10} that the set $\SZ$ is $\mathfrak{c}^+$-lineable and, also, $\mathfrak{c}$-algebrable. As a consequence, assuming that $\mathfrak{c}^+ = 2^{\mathfrak{c}}$ (which follows, for instance, from the Generalized Continuum Hypothesis or GCH), \(\SZ\) would be  \(2^\mathfrak{c}\)-lineable.
\item Also, in 2010 \cite[Corollary 2.11]{GMS_2010}, G\'{a}mez-Merino, Mu\~{n}oz-Fern\'{a}ndez, and Seoane-Sep\'{u}lveda proved that $\SZ$ is actually \(d_{\mathfrak c}\)-lineable, where \(d_{\mathfrak c}\) is a cardinal invariant defined as
    \[
    d_{\mathfrak c}=\min\{\,\card F\,:\,F\subset\mathbb R^{\mathbb R},(\forall \varphi\in\mathbb R^{\mathbb R})(\exists f\in F)(\card(f\cap \varphi)=\mathfrak c)\,\}.
    \]
This cardinal can take as value any regular cardinal between \(\mathfrak c^+\) and \(2^{\mathfrak c}\), depending of the set-theoretical axioms assumed.

\item Later, in \cite[Theorem 2.6]{BGPS_PAMS}, Bartoszewicz, G\l \c ab, Pellegrino, and Seoane-Sep\'{u}lveda showed that \(\SZ\) is actually  \(\kappa\)-strongly algebrable for some $\mathfrak{c}^+ \le \kappa \le 2^{\mathfrak{c}}$ if there is in \(\mathfrak c\) an almost disjoint family of cardinality \(\kappa\) (see Definition \ref{almostdisjoint} below). Assuming either Martin's Axiom, or~CH, or $2^{<\mathfrak c}=\mathfrak c$, this \(\kappa\) can be chosen to be \(2^{\mathfrak c}\), so we would have that $\SZ$ is \(2^{\mathfrak{c}}\)-strongly algebrable.

\end{enumerate}

So far, and as we can see in the previous background, a lot of effort has been invested in trying to achieve the $2^{\mathfrak c}$-lineability (maximal lineability) of $\SZ$ without the need of any additional set theoretical assumptions and, still, the problem remains open. In this note we shall prove that it is, actually, undecidable and, in order to obtain this result we first need the study the notion of almost disjoint families, which is the topic of the next section.

\section{The relation between Sierpi\'ski-Zygmund functions and almost disjoint families}

The following is a well known concept in Set Theory (see, e.g. \cite{EHM_1968,K_1983}).

\begin{definition}\label{almostdisjoint}
Let \(S\) a subset of cardinality \(\kappa\). We say that a family \(\mathcal F\subset\mathcal P(S)\) is an \emph{almost disjoint family} in \(S\) if the following conditions hold:
\begin{enumerate}
\item If \(A\in\mathcal F\) then \(\card A=\kappa\).
\item If \(A,B\in\mathcal F\), \(A\neq B\), then \(\card(A\cap B)<\kappa\).
\end{enumerate}
\end{definition}

As we already mentioned earlier, it is still not known whether any additional set-theoretical assumptions are needed or not in order to show the $2^\mathfrak{c}$-strongly algebrability (and the $2^\mathfrak{c}$-lineability) of $\SZ$. Our next result shows that solving this question is equivalent to the existence,  in \(\mathfrak c\), of an almost disjoint family of cardinality \(2^{\mathfrak c}\).

\begin{theorem}
Let \(\kappa\) be a cardinal number such that $\mathfrak c \le \kappa \le 2^{\mathfrak{c}}$. The following are equivalent:
\begin{enumerate}
\item\label{item_str_alg} $\SZ$ is \(\kappa\)-strongly algebrable.
\item\label{item_alg} $\SZ$ is \(\kappa\)-algebrable.
\item\label{item_lin} \(\SZ\) is \(\kappa\)-lineable.
\item\label{al_disj} There exists in \(\mathfrak c\) an almost disjoint family of cardinality \(\kappa\).
\end{enumerate}
\end{theorem}

\begin{proof}{\ }\par
(\ref{item_str_alg}\(\Rightarrow\)\ref{item_alg}).\quad Obvious.

(\ref{item_alg}\(\Rightarrow\)\ref{item_lin}).\quad Obvious.

(\ref{item_lin}\(\Rightarrow\)\ref{al_disj}).\quad Let us assume that \(V\subset\SZ\cup\{0\}\) is a \(\kappa\)-dimensional vector space. To prove that there exists in \(\mathfrak c\) an almost disjoint family of cardinality \(\kappa\), it will suffice to find a family like this in some set of cardinal~\(\mathfrak c\). We shall prove that \(V\) (whose cardinality is \(\kappa\)) is an almost disjoint family in~\(\mathbb R^2\).

If \(f\in V\) it is quite obvious that \(\card f=\mathfrak c\), so the first condition in the definition is accomplished.

To prove the second condition, let \(f,g\in V\), \(f\neq g\). Then we can not have \(\card(f\cap g)=\mathfrak c\), because in that case \(f-g=0\) in a set of cardinality \(\mathfrak c\), and therefore the restriction of \(f-g\) to that set would be continuous. This is impossible, because \(f-g\in V\setminus\{0\}\subset\SZ\). So, we must have \(\card(f\cap g)<\mathfrak c\).

(\ref{al_disj}\(\Rightarrow\)\ref{item_str_alg}).\quad This is just \cite[Theorem 2.6]{BGPS_PAMS}.
\end{proof}

Although we have not yet solved the problem of the $2^\mathfrak{c}$-lineability of the set $\SZ$, the next section shall give the ultimate answer to this open question thanks to the previous theorem and some set theoretical techniques.

\section{The size of an almost disjoint family in \(\mathfrak c\)}

Let us, next, review a series of results on almost disjoint families, all of which can be found in \cite{K_1983}.

\begin{remark}\
\begin{enumerate}
\item On the one hand, recall that under ZFC there is an almost disjoint family of cardinality \(\mathfrak c=2^{\aleph_0}\) in \(\aleph_0\).
\item\label{enu:alephsubone} On the other, the existence of an almost disjoint family of cardinality \(2^{\aleph_1}\) in \(\aleph_1\) is undecidable.
\item\label{enu:consistent} Also, and under the set-theoretical assumption \(2^{<\mathfrak c}=\mathfrak c\), there exists an almost disjoint family of cardinality \(2^{\mathfrak c}\) in \(\mathfrak c\).
\end{enumerate}
\end{remark}

Let us point out that (\ref{enu:consistent}) says that is consistent with ZFC that \(\mathcal SZ\cup\{0\}\) contains a vector space of dimension \(2^{\mathfrak c}\). We shall see in the following that the contrary is also consistent. The proof follows, roughly, the lines of that of (\ref{enu:alephsubone}) (see \cite{K_1983}*{p.~290 (B4)}).

\begin{theorem}
In some model of ZFC there is no almost disjoint family in \(\mathfrak c\) whose cardinality is~\(2^{\mathfrak c}\).
\end{theorem}

\begin{proof}
Let us take a model \(M\) of \(ZFC+GCH\) as ground model. Let \(\mathbb P\in M\) be an Easton forcing obtained from an index function \(E(\aleph_0)=\aleph_2\), \(E(\aleph_1)=\aleph_4\); see \cite{K_1983}*{Ch.~VIII, \S4}. (This forcing is equivalent to the iteration of \(\operatorname{Fn}(\aleph_4\times\aleph_1,2,\aleph_1)\) and \(\operatorname{Fn}(\aleph_2\times\aleph_0,2,\aleph_0)\) (\cite{K_1983}*{Lemma~VIII~4.3})). 

Let \(G\) be a generic filter for \(\mathbb P\). In the generic extension \(M[G]\) we have (\cite{K_1983}*{Theorem~VII~4.7}) that \(\mathfrak c=2^{\aleph_0}=\aleph_2\), \(2^{\aleph_1}=\aleph_4\), and also \(2^{\mathfrak c}=2^{\aleph_2}=\aleph_4\). We shall see that in this generic extension, there is no almost disjoint family of cardinality \(\aleph_4\) in~\(\omega_2\). Indeed, suppose that some \(p\in\mathbb P\) forces the existence of a family of \(\aleph_4\) almost disjoint subsets of \(\omega_2\). Then there would be \(\mathbb P\)-names \(\dot E_\alpha\) for \(\alpha<\omega_4\) such that \(p\) forces that each \(\dot E_\alpha\subset\omega_2\) and that \(\card(\dot E_\alpha\cap\dot E_\beta)<\aleph_2\), whenever \(\alpha<\beta\). By \cite{K_1983}*{Lemma~VIII~4.4}, \(\mathbb P\) has the \(\aleph_2\)-cc. Therefore, using \cite{K_1983}*{Lemma~VIII~5.6}, whenever \(\alpha<\beta\), there is a \(\gamma_{\alpha,\beta}<\omega_2\) such that \(p\)\/ forces that \(\dot E_\alpha\cap \dot E_\beta\subset\gamma_{\alpha,\beta}\). 

Next, using the \((2^{\aleph_2})^+\to(\aleph_3)_{\aleph_2}^2\) instance of the Erd\H{o}s-Rado Partition Theorem (see \cite{K_1983}*{p.~290 (B1)}), which is equivalent to \(\aleph_4\to(\aleph_3)_{\aleph_2}^2\) because GCH holds in \(M\), we have that there exist a subset \(H\subset\omega_2\) such that \(\card H=\aleph_3\) and \(\gamma<\omega_2\) such that \(p\) forces that \(\dot E_\alpha\cap\dot E_\beta\subset\gamma\) whenever \(\alpha,\beta\in H\), \(\alpha<\beta\). If we define \(F_\alpha=\dot E_\alpha\setminus\gamma\) for every \(\alpha\in H\) we have:
\begin{enumerate}
\item For every \(\alpha\in H\), \(p\) forces that \(F_\alpha\subset\omega_2\).
\item For every \(\alpha\in H\), \(p\) forces that \(F_\alpha\neq\varnothing\), because \(\card E_\alpha=\aleph_2\) and \(\card \gamma<\aleph_2\).
\item If \(\alpha,\beta\in H\), \(\alpha<\beta\), then \(p\) forces that \(F_\alpha\cap F_\beta=\varnothing\).
\end{enumerate}
Thus, we get a contradiction, because in \(M[G]\) the family \(\{F_\alpha\}_{\alpha\in H}\) is a \emph{pairwise disjoint} family of \(\aleph_3\) many elements in \(\omega_2\).
\end{proof}

Hence, we obtain what the title of this note states:

\begin{corollary}
The $2^{\mathfrak{c}}$-lineability (maximal lineability) of the set of Sierpi\'{n}ski-Zygmund functions is undecidable.
\end{corollary}

%%%%%%%%%%%
%%%%%%%%%%%
%%%%%%%%%%%
%%%%%%%%%%%
% REFERENCES
%%%%%%%%%%%
%%%%%%%%%%%
%%%%%%%%%%%
%%%%%%%%%%%

%%%%%%%%%%%%%%%%
%%%%%%%%%%%%%%%%
%%%%%%%%%%%%%%%%
\end{document}